\documentclass[10pt]{amsart}
\usepackage[dvips]{graphicx}
\usepackage[latin1]{inputenc}
\usepackage{amsmath,amsthm,amsopn,amstext,amscd,amsfonts,amssymb}
\usepackage{amssymb}
\usepackage{amsfonts}
\usepackage{lscape}
\long
\def\salta#1{\relax}
\newcommand{\R}{{I\!\!R}}
\newcommand{\N}{{I\!\!N}}

\newcommand{\car}{{\raise2pt\hbox{$\chi$}}}

\newcommand{\dob}{{\mathcal D}^{1,2}}

\renewcommand{\a }{\alpha }

\renewcommand{\t }{\tau }

\newtheorem{Theorem}{Theorem}[section]
\newtheorem{Corollary}[Theorem]{Corollary}

\newtheorem{Lemma}[Theorem]{Lemma}

\newcommand{\B}{\mathcal{B}}
\newcommand{\C}{\mathcal{C}}
 \newcommand{\h}{\mathcal{H}}

\newcommand{\F}{\mathcal{F}}
 \newcommand{\W}{\mathcal{W}}

\begin{document}
\title[A decomposition result for a singular elliptic  Equation.]{A decomposition result for a singular elliptic equation on
compact Riemannian manifolds.}
\author[Y. Maliki, F.Z. Terki]{Y. Maliki$^*$ and F.Z. Terki}
\address{Y. Maliki, F.Z. Terki \hfill \break\indent D\'epartement de
Math\'ematiques, Universit\'e Abou Bakr Belkaïd, Tlemcen, \hfill\break%
\indent Tlemcen 13000, Algeria.} \email{\texttt{malyouc@yahoo.fr,
fatimazohra113@yahoo.fr}}

\date{}
\maketitle

\begin{abstract} On compact Riemannian manifolds,
we prove a decomposition theorem for arbitrarily bounded energy
sequence of solutions of a singular elliptic equation.
\end{abstract}

\section{Introduction}

Let $(M,g$) be an $(n\geq 3)-$dimensional Riemannian manifold. In
this paper, we are interested in studying on $(M,g)$ the asymptotic
behaviour of a sequence of solutions $u_{\alpha }$, when
$\alpha\to\infty$, of the following singular elliptic equation:
\begin{equation} \label{eq1.1}
\Delta _{g}u-\frac{h_{\alpha }}{\rho_p ^{2}(x)}u=f(x)|u|^{2^*-2}u,
\tag{$E_{\alpha }$}
\end{equation}
where $2^*=\frac{2n}{n-2}$, $h_{\alpha }$ and $f$ are functions on
$M$, $p$ is a fixed point of $M$ and
$\rho_p(x)=dist_g(p,x)$ is the distance function on $ M$ based at $p$ ( see definition \eqref{eq2.0}).\\
Certainly, if the singular term  $\frac{h_{\alpha }}{\rho_p ^{2}(x)}
$ is replaced by $\frac{n-2}{4(n-1)}Scal_g $, then equation
$E_{\alpha }$ becomes the famous prescribed scalar curvature
equation which is very known in the literature. When $f$ is constant
and the function $\rho_p$ is of power $0<\gamma<2$, equation
\eqref{eq1.1} can be seen as a case of
 equations that arise in the study of conformal deformation to constant scalar
curvature of metrics which are smooth only in some ball
$B_p(\delta)$ ( see \cite{Madani}).\\
 Equations of type
\eqref{eq1.1} have  been the subject of interest especially on the
Euclidean space $\R^n$. A famous result has been obtained in
\cite{Terracini} and it consists of the classification of positive
solutions of the equation
\begin{equation}\label{eq0.1}
\Delta u -\frac{\lambda}{|x|^2}=u^{\frac{n+2}{n-2}}, \tag{$E$}
\end{equation}
where $ 0<\lambda<\frac{(n-4)^2}{4}$, into the family of functions
\begin{equation*}
u_{\lambda }(x)=C_{\lambda }\left(\frac{\left\vert x\right\vert ^{a-1}}{%
1+\left\vert x\right\vert ^{2a}}\right)^{\frac{n}{2}-1}.
\end{equation*}
 where $c_\lambda$ is some constant and $a=\sqrt{1-\frac{4\lambda}{(n-2)^{2}}}.$\\
In terms  of decomposition of Palais-Smale sequences of functional
energy, this family of solutions was employed in \cite{D. Smet} in
constructing singularity bubbles,
\begin{equation*}
\B_\lambda^{\varepsilon_\alpha,y_\alpha}=
\varepsilon_\alpha^{\frac{2-n}{2}}u_\lambda(\frac{x-y_\alpha}{\varepsilon_\alpha}),\text{
with } \frac{|y_\alpha|}{\varepsilon_\alpha}\to0,
\end{equation*}
which, together with the  classical bubbles caused by the existence
of critical exponent
$$ \B_0^{\varepsilon_\alpha,y_\alpha}=
\varepsilon_\alpha^{\frac{2-n}{2}}u_0(\frac{x-y_\alpha}{\varepsilon_\alpha}),\text{
with } \frac{|y_\alpha|}{\varepsilon_\alpha}\to\infty,$$ where $u_0$
being the solution of the non perturbed equation $\Delta
u=u^{\frac{n+2}{n-2}}$  give a whole  picture of the decomposition
of the Palaise-Smale sequences. This decomposition result has been
proved in \cite{D. Smet} and was the key component for the obtention
of interesting existence results for equation \eqref{eq0.1} with a
function $K$ get involved in the nonlinear term. Similar
decomposition result has been obtained in \cite{D.Cao} for equation
\eqref{eq0.1} with small perturbation, the authors described
asymptotically the associated Palais-Smale sequences of bounded energy.\\
 The compactness result obtained
in this paper  can be seen as an extension to Riemannian context of
those obtained in \cite{D. Smet} and \cite{D.Cao} in the Euclidean
context, the difficulties when working in the Riemannian setting
reside mainly in the construction of bubbles.\\
Historically, a famous compactness result for elliptic  value
problems on domains of $\mathbb{R}^n$ has been obtained by M.Struwe
in \cite{Struwe}. Struwe's result has been extended later by O.Druet
et al. in \cite{Druet-hebbey-robert} to elliptic equations on
Riemannian manifolds like
\begin{equation*}
\Delta _{g}u+h_{\alpha }u=u^{2^*-1}.
\end{equation*}
Many results have been obtained by the authors describing the
asymptotic behaviour of Palais-Smale sequences. The authors gave a
detailed construction of bubbles by means of a re-scaling process
via the exponential map at some points, supposed to be the centers
of bubbles. The author in \cite{M.Dellinger} followed the same
procedure to prove a decomposition result on compact Riemannian
manifolds for a Sobolev-Poincaré equation.\\
For our case, we will use, when necessary, ideas from
\cite{Druet-hebbey-robert} to prove a decomposition theorem for
equation \eqref{eq1.1}. More explicitly, after determining
conditions under which solutions of \eqref{eq1.1} exist, we prove as
in \cite{D. Smet} and \cite{D.Cao} that, under some conditions on
the sequence $h_\alpha$ and the function $f$,  a sequence of
solutions of \eqref{eq1.1} of arbitrarily  bounded energy decomposes
into the sum of a solution of the the limiting equation
\begin{equation}\label{0.1}\tag{$E_{\infty }$}
\Delta _{g}u-\frac{h_{\infty }(p)}{\rho_p
^{2}(x)}u=f(p)|u|^{2^*-2}u,
\end{equation}
where $h_\infty$ is the uniform limit of $h_\alpha$,  and two kinds
of bubbles, namely  the classical and the singular ones due to the
presence respectively of the critical exponent and the singular
term.
\section{Notations and preliminaries }
In this section, we introduce some notations and materials necessary
in our study. Let $H_{1}^{2}(M)$ be the Sobolev space consisting of
the completion of $\C^{\infty}(M)$ with respect to the norm
$$||u||_{H_{1}^{2}(M)}=\int_M(|\nabla u|^2+u^2)dv_g.$$
$M$ being compact, $H_{1}^{2}(M)$ is then  embedded in $L_q(M)$
compactly for $q<2^*=\frac{2n}{n-2}$ and continuously for $q=2^*$.\\
Let $K(n,2)$ denote the best constant in Sobolev inequality that
asserts that there exists a constant $B>0$ such that for any $u\in
H_{1}^{2}(M)$,
\begin{equation}\label{1.1}
 ||u||^2_{L_{2^*}(M)}\le K^2(n,2)||\nabla u||^2_{L_2(M)}+B||u||^2_{L_2(M)}.
\end{equation}
Throughout the paper, we will denote by $B(a,r)$ a ball of center
$a$ and radius $r>0$, the point $a$ will be specified either in $M$
or in $\R^n$, and $B(r)$ is a ball in $\R^n$ of center $0$ and
radius $r>0$.\\
Denote by $\delta_g$ the injectivity radius of $M$. Let $p\in M$ be
a fixed point, as in \cite{Madani} we define the function $\rho_p$
on $M$ by
\begin{equation}\label{eq2.0}
    \rho_p(x)=\left\{
       \begin{array}{ll}
         dist_g(p,x), & dist_g(p,x)< \delta_g,\\
         \delta_g, & dist_g(p,x)\ge\delta_g
       \end{array}
     \right.
\end{equation}
For $q\ge1$, we denote by $L_q(M,\rho_p^{\theta})$ the space of
functions $u$ such that $\frac{u}{\rho_p^{\theta}}$ is integrable.
This space is endowed with norm
$\|u\|^q_{q,\rho_p^{\theta}}=\int_M\frac{|u|^q}{\rho_p^{\theta}}dv_g$.\\
In \cite{Madani}, the following Hardy inequality has been proven on
any compact manifold $M$, for every $\varepsilon>0$ there exists a
positive constant $A(\varepsilon)$ such that for any $u\in
H^2_1(M)$,
\begin{equation}\label{2.1}
\int_M \frac{u^2}{\rho^2_p}dv_g\le(
K^2(n,2,-2)+\varepsilon)\int_M|\nabla
u|^2dv_g+A(\varepsilon)\int_Mu^2dv_g,
\end{equation}
 with $K(n,2,-2)$ being the
best constant in the Euclidean Hardy inequality
\begin{equation*}
\int_{\mathbb{R}^n} \frac{u^2}{|x|^2}dx \le
K(n,2,-2)^2\int_{\mathbb{R}^n}|\nabla u|^2dx,
u\in\C^\infty_o(\mathbb{R}^n).
\end{equation*}
 If $u$ is supported in a ball $B(p,\delta),0<\delta<\delta_g$,  then
\begin{equation*}
\int_{B(p,\delta)} \frac{u^2}{\rho^2_p}dv_g\le
K_\delta(n,2,-2)\int_{B(p,\delta)}|\nabla u|^2dv_g,
\end{equation*}
with $ K_\delta(n,2,-2)$ goes to $K(n,2,-2)$ when $\delta$ goes to
$0$. \\
Concerning the existence of solutions of equations \eqref{eq1.1},
the author in \cite{Madani} proved through the classical variational
techniques an existence result with $f$  a constant function. By
following the same procedure, though the presence of the non
constant function $f$ adds further technical difficulties, we can
prove the existence of a non trivial  weak solution of
\eqref{eq1.1}. This existence result is formulated in the following
theorem and due to the very
familiarity of the techniques used, we omit the proof.\\
 For $u\in H^2_1(M)$, set
\begin{equation*}
\mu=\inf_{u\in H^2_1(M),u\neq0}\frac{\int_M(|\nabla
u|^2-\frac{h}{\rho_p^2}u^2)dv_g}{(\int_Mf|u|^{2^*}dv_g)^{\frac{2}{2^*}}}.
\end{equation*}
The following theorem ensures conditions under which a weak solution
$u_\alpha$ of \eqref{eq1.1} exists.
\begin{Theorem}
\label{exist0} Let $(M,g)$ be a compact $n(n\geq 3)-$dimensional
Riemannian manifold and $f,h_\alpha(\alpha\in[0,\infty[)$ be
continuous functions on $M$. Under the following conditions :
\begin{enumerate}
\item $0<h_\alpha(p)<\frac{1}{K^{2}(n,2,-2)}$
\item $f(x)>0 ,\forall x\in M$ and $\mu<\frac{1-h_\alpha(p)K^{2}(n,2,-2)}{(\sup_Mf)^{\frac{n-2}{n }}K^{2}(n,2)},$
\end{enumerate}
equation \eqref{eq1.1} admits a nontrivial weak solution
$u_\alpha\in H^2_1(M)$.
\end{Theorem}

\section{Decomposition theorem}
Let $J_\alpha$ be the functional defined on $H_1^2(M)$ by
$$J_\alpha(u)=\frac{1}{2}\int_M(|\nabla
u|^2-\frac{h_\alpha}{\rho^2}
u^2)dv_g-\frac{1}{2^*}\int_Mf|u|^{2^*}dv_g.$$ Traditionally, we
define a Palais-Smale sequence $v_\alpha$ of $J_\alpha $ at a level
$\beta$ as to be the sequence that satisfies $J_\alpha(v_\alpha)\to
\beta $ and $DJ_\alpha(v_\alpha)\varphi\to 0,\forall\varphi \in H^2_1(M)$.\\
Define the following limiting functionals
\begin{eqnarray*}
  J_{\infty}(u) &=&  \frac{1}{2}(\int_M(|\nabla
u|^2-\frac{h_\infty}{\rho^2}
u^2)dv_g-\frac{1}{2^*}\int_Mf|u|^{2^*}dv_g,  u\in H^2_1(M)\\
G(u)&=& \frac{1}{2}\int_{\R^n}|\nabla
u|^2dx-\frac{1}{2^*}\int_{\R^n}|u|^{2^*}dx,u\in D^{1,2}(\R^n), \text{ and } \\
 G_\infty (u)&=& \frac{1}{2}\int_{\R^n}|\nabla
u|^2dx- \frac{h_\infty(p)}{2}\int_M\frac{u^2}{|x|^2}
dx-\frac{f(p)}{2^*}\int_{\R^n}|u|^{2^*}dx,u\in D^{1,2}(\R^n)
\end{eqnarray*}
For $\alpha\in[0,\infty[$, let $h_\alpha$ be a sequence of
continuous functions on $M$ such that
\begin{quote}
\begin{equation*}
(\h)\left\{
  \begin{array}{ll}
   &\text{a- }|h_\alpha(x)|\le C, \text{ for some constant }  C>0, \forall x\in
M \text{ and } \forall\alpha\in[0,\infty[. \\&
   \text{b- There exists a function }  \text{ such that }\sup_M|h_\alpha-h_\infty|\to
0,\\&
  \text{c- } 0<h_\alpha(p)<\frac{1}{K^{2}(n,2,-2)}, \text{ for all } \alpha, 0\le\alpha\le\infty. \\
  \end{array}
\right.
\end{equation*}
\end{quote}
Now, we state our main result
\begin{Theorem} Let $(M,g)$ be a Riemannian manifold with
$dim(M)=n\ge3$, $h_\alpha$ be  a sequence of continuous functions on
$M$ satisfying $(\h)$,  $f$ be  a positive continuous function on
$M$ that satisfies with $h_\alpha$ the conditions of theorem 2.1.
Let $u_\alpha$ be a sequence of weak  solutions of\eqref{eq1.1} such
that $\int_Mf|u_\alpha|^{2^*}dv_g\le C,\forall\alpha>0$. Then, there
exist $m \in \N$, sequences $
R_\alpha^i>0,R_\alpha^i\underset{\alpha\to\infty}{\to}0$, $k\in\N^n$
sequences
$\t_\alpha^j>0,\t_\alpha^j\underset{\alpha\to\infty}{\to}0$,
converging sequences  $x_\alpha^j\to x_o^j\neq p$ in $M$, a solution
$u_o\in H^2_1(M)$ of \eqref{0.1},  solutions $v_i\in D^{1,2}(\R^n)$
 of \eqref{4.15} and nontrivial solutions $\nu_j\in
D^{1,2}(\R^n)$ of \eqref{eq3.8} such that up to a subsequence
\begin{eqnarray*}
u_\alpha&=&u_o+\sum_{i=1}^{k}(R^i_\alpha)^{\frac{2-n}{n}}
\eta_\delta(\exp^{-1}_p(x))v_i((R_\alpha^i)^{-1}\exp^{-1}_p(x))\\&+&\sum_{j=1}^{l}(r^i_\alpha)^{\frac{2-n}{n}}
f(x_o)^{\frac{2-n}{4}}\eta_\delta(\exp^{-1}_{x_\alpha^j}(x))\nu_j((r_\alpha^j)^{-1}\exp^{-1}_{x_\alpha^j}(x))+\W_\alpha,\\&&
\text{ with } \W_\alpha\to 0 \text{ in }H^1_2(M),
\end{eqnarray*}
and
\begin{equation*}
J_\alpha(u_\alpha)=J_{\infty}(u_o)+\sum_{i=1}^k
G_\infty(v_i)+\sum_{j=1}^lf(x_o^j)^{\frac{2-n}{2}} G(\nu_j)+o(1).
\end{equation*}
\end{Theorem}
In order to prove this theorem, we prove some useful lemmas. In all
what follows, $h_\alpha$ is supposed to satisfy conditions $(\h)$.
\begin{Lemma} Let $u_\alpha$ be a Palais-Smale sequence for
$J_\alpha$ at level $\beta$  that converges to a function $u$ weakly
in $H^2_1(M)$ and $L_2(M,\rho_p^2)$, strongly in $L_q(M), 1\le
q<2^*$ and almost everywhere to a function $u$ . Then, the sequence
$v_\alpha=u_\alpha-u$ is sequence of Palais-Smale for $J_\alpha$ and
$$J_\alpha(v_\alpha)=\beta-J_\infty(u)+o(1).$$
\end{Lemma}
\begin{proof}
First, in view of the fact that $u_\alpha$ is a Palais-Smale
sequence for $J_\alpha$, $u_\alpha$  is bounded in $H^2_1(M)$. In
fact, $DJ_\alpha(u_\alpha)u_\alpha=o(||u||_{H^2_1(M)})$ implies that
\begin{equation*}
J_\alpha(u_\alpha)=\frac{1}{n}\int_Mf|u_\alpha|^{2^*}dv_g=\beta+o(1)+o(||u||_{H^2_1(M)}).
\end{equation*}
 Since $f>0$, this implies in turn that $u_\alpha$ is bounded in
$L_{2^*}(M)$ and then in $L_2(M)$. Furthermore, we have
\begin{equation*}
\int_M|\nabla u_\alpha|^2dv_g=
nJ_\alpha(u_\alpha)+\int\frac{h_\alpha}{\rho_p^2}u_\alpha^2dv_g+o(||u||_{H^2_1(M)})
\end{equation*}
By continuity of $h_\alpha$ on $p$, we have that for all
$\epsilon>0$ there exists $\delta>0$ such that
\begin{eqnarray*}
&\int_M|\nabla u_\alpha|^2dv_g\le
n\beta+(\varepsilon+h_\alpha(p))\int_{B(p,\delta)}\frac{h_\alpha}{\rho_p^2}u_\alpha^2dv_g
&\\&+\delta^{-2}\int_{M\setminus B(p,\delta)}h_\alpha
u_\alpha^2dv_g+o(||u||_{H^2_1(M)})+o(1),&
\end{eqnarray*}
then, by applying Hardy inequality \eqref{2.1} that for every
$\varepsilon>0$ small there exists a constant $A(\varepsilon)$ such
that
\begin{eqnarray*}
&\int_M|\nabla u_\alpha|^2dv_g\le
n\beta+(\varepsilon+h_\alpha(p))(\varepsilon+
K^2(n,2,-2))\int_{M}|\nabla u_\alpha|^2dv_g
&\\&+A(\varepsilon)\int_{M} u_\alpha^2dv_g+o(||u||_{H^2_1(M)})+o(1)&
\end{eqnarray*}
since $0<h_\alpha(p)<\frac{1}{K^2(n,2,-2)}$, we can find
$\varepsilon >0$ small such that
$1-(\varepsilon+h_\alpha(p))(\varepsilon+ K^2(n,2,-2))>0$ which
implies that $\int_M|\nabla u_\alpha|^2dv_g$ is bounded. Thus, $u_\alpha $ bounded in $H^2_1(M)$.\\
Now, for two functions $\varphi,\phi\in H_1^2(M)$, Hölder and Hardy
inequalities give
\begin{equation}\label{eqn3.1}
\int_M|\frac{h_\alpha-h_\infty}{\rho_p^2}\phi\varphi| dv_g
   \le C ||\varphi||_{H_1^2(M)}||\phi||_{H^2_1(M)}\sup_M|h_\alpha-h_\infty|,
\end{equation}
 writing
\begin{equation*}
\int_M\frac{h_\alpha}{\rho_p^2}\phi\varphi
dv_g=\int_M\frac{h_\alpha-h_\infty}{\rho_p^2}\phi\varphi dv_g+
\int_M\frac{h_\infty}{\rho_p^2}\phi\varphi dv_g,
\end{equation*}
we get by the assumption made on the sequence $h_\alpha$ that
\begin{equation}\label{eqn3.2}
\int_M\frac{h_\alpha}{\rho^2_p}\phi\varphi dv_g=
\int_M\frac{h_\infty}{\rho^2_p}\phi\varphi dv_g+o(1).
\end{equation}
Then, since the sequence $u_\alpha$ is bounded in $H^2_1(M)$, by
taking $\phi=u_\alpha$, we get from \eqref{eqn3.1} together with the
weak convergence of $u_\alpha$ to $u$ in $L^2(M,\rho^{-2})$ that
\begin{equation}\label{eq3.3}
\int_M\frac{h_\alpha}{\rho_p^2}u_\alpha\varphi dv_g=
\int_M\frac{h_\infty}{\rho^2_p}u\varphi dv_g+o(1),
\end{equation}
thus, applying the last identity to $\varphi=u$, we get by the weak
convergence in $H_1^2(M)$ that
\begin{equation*}
 J_\alpha(v_\alpha)=J_\alpha(u_\alpha)-J_\infty(u)+\Phi(u_\alpha)+o(1),
 \end{equation*}
 with
 $$\Phi_\alpha(u_\alpha)=\frac{1}{2^*}\int_Mf(|u_\alpha|^{2^*}-|u|^{2^*}-|v_\alpha|^{2^*})dv_g,$$
 which by the Brezis-Leib convergence Lemma equals to $o(1)$, hence we obtain
\begin{equation*}
J_\alpha(v_\alpha)=\beta-J_\infty(u)+o(1).
\end{equation*}
Moreover, for $\varphi\in
 H_1^2(M)$, by taking $\phi=u$ in \eqref{eqn3.2}, we can write
\begin{equation*}
DJ_\alpha(v_\alpha)\varphi=DJ_\alpha(u_\alpha)\varphi-DJ_\infty(u)\varphi+
\Phi(v_\alpha)\varphi+o(1),
\end{equation*}
 with
  \begin{eqnarray*}
\Phi(v_\alpha)\varphi&=&\int_Mf\left(|v_\alpha+u|^{2^*-2}(v_\alpha+u)-
|v_\alpha|^{2^*-2}v_\alpha-|u|^{2^*-2}u\right)\varphi dv_g.
\end{eqnarray*}
Knowing that there exists a positive constant $C$ such that
\begin{equation*}
\mid|v_\alpha+u|^{2^*-2}(v_\alpha+u)-
|v_\alpha|^{2^*-2}v_\alpha-|u|^{2^*-2}u\mid\le
C(|v_\alpha|^{2^*-2}|u|+|u|^{2^*-2}|v_\alpha|),
\end{equation*}
 we get, after applying Hölder inequality, that there exists a  positive constant $C$ such that
\begin{equation*}
|\Phi(v_\alpha)\varphi|\le C
\left(\||v_\alpha|^{2^*-2}|u|\|_{L_{\frac{2^*}{2^*-1}}(M)}+
\||u|^{2^*-2}|v_\alpha|\|_{L_{\frac{2^*}{2^*-1}}(M)}\right)\|\varphi\|_{L_{2^*}(M)},
\end{equation*}
which gives that $\Phi(u_\alpha)=o(1)$ since both
$\frac{2^*(2^*-2)}{2^*-1}$ and $ \frac{2^*}{2^*-1}$ are smaller that
$2^*$ and the inclusion of $H^2_1(M)$ in $L_q(M)$ is compact for $q<2^*$.\\
On the other hand, since the sequence $u_\alpha^{2^*-2}u_\alpha$ is
bounded in $L_{\frac{2^*}{2^*-1}}(M)$ and converges almost
everywhere to $u^{2^*-2}u$ , we get that $u_\alpha^{2^*-2}u_\alpha$
converges weakly in $L_{\frac{2^*}{2^*-1}}(M)$ to $u^{2^*-2}u$.
This, together with the weak convergence in $H^2_1(M)$ of $u_\alpha$
to $u$ and relation \eqref{eq3.3}, imply that $
DJ_\infty(u)\varphi=0,\forall\varphi \in H^2_1(M)$. Hence,
$DJ_\alpha(v_\alpha)\varphi\to 0, \forall\varphi \in H^2_1(M)$.
\end{proof}
\begin{Lemma} Let $v_\alpha$ be a Palais-Smale sequence of $J_{\alpha}$ at level $\beta$ that converges
weakly to $0$ in $H_1^2(M)$. If
$\beta<\beta^*=\frac{\left(1-h_\infty(p)K^2(n,2,-2)\right)^{\frac{n}{2}}}{n(\sup_M
f)^{\frac{n-2}{2}}K(n,2)^n}$, then  $v_\alpha$ converges strongly to
$0$ in $H_1^2(M)$.
\end{Lemma}
\begin{proof} If $v_\alpha$ is a Palais-Smale sequence of $J_{\alpha}$ at level
$\beta$ that converges to $0$ weakly in $H_1^2(M)$, then
$\int_Mu_\alpha^2dv_g=o(1)$ and
\begin{equation*}
\beta=\frac{1}{n}\int_M(|\nabla
v_\alpha|^2-\frac{h_\alpha}{\rho^2_p}v_\alpha^2)dv_g=\frac{1}{n}\int_Mf|v_\alpha|^{2^*}dv_g+o(1).
\end{equation*}
This implies that $\beta\ge0$. Hence, on the one hand, by Hardy
inequality \eqref{2.1} we get as in Lemma 3.2, that for small enough
$\varepsilon>0$,
\begin{equation}\label{eqn3.3}
\int_M|\nabla
v_\alpha|^2dv_g\le\frac{n\beta}{1-[(h_\alpha(p)+\varepsilon)(\varepsilon+
K^2(n,2,-2))]}+o(1),
\end{equation}
and on the other hand, by Sobolev inequality \eqref{1.1}, we also
get
\begin{equation}\label{eqn3.4}
\int_M|\nabla v_\alpha|^2dv_g \ge \left(\frac{n\beta}{(\sup_M
f)K^{2^*}(n,2)}\right)^{\frac{2}{2^*}}+o(1).
\end{equation}
Now, suppose that $\beta>0$, then the above inequalities
\eqref{eqn3.3} and \eqref{eqn3.4} , for  $\alpha$ big  enough, give
\begin{equation*}
\beta\ge
\frac{\left(1-(h_\infty(p)+2\varepsilon)(K^2(n,2,-2)+\varepsilon)\right))^{\frac{n}{2}}}{n(\sup_M
f)^{\frac{n-2}{2}}K(n,2)^n},
\end{equation*}
that is
\begin{equation*}
    \beta^{\frac{2}{n}}\ge{\beta^{*}}^{\frac{2}{n}}-\frac{2\varepsilon^2+\varepsilon(h_\infty(p)+2\varepsilon
K^2(n,2,-2))}{n^{\frac{2}{n}}(\sup_M f)^{\frac{n-2}{n}}K(n,2)^2}.
\end{equation*}
By  assumption $\beta^*>\beta$, by taking $\varepsilon>0$ small
enough so that
\begin{equation*}
-2\varepsilon^2-\varepsilon(h_\infty(p)-2\varepsilon
K^2(n,2,-2))+n^{\frac{2}{n}}(\sup_M
f)^{\frac{n-2}{n}}K(n,2)^2({\beta^*}^{\frac{2}{n}}-\beta^{\frac{2}{n}})>0,
\end{equation*}
we get a contradiction. Thus $\beta=0$ and \eqref{eqn3.3} assures
that
\begin{equation*}
\int_M|\nabla v_\alpha|^2dv_g=o(1),
\end{equation*}
that is $v_\alpha\to0$ strongly in $H_1^2(M)$.
\end{proof}
In the following, for a given positive constant $R$, define a
cut-off function $\eta_R\in C_o^{\infty}( \R^n)$ such that
$\eta_R(x)=1, x\in B(R)$ and
$\eta_R(x)=0,x\in \R^n\setminus B(2R)$,$0\le\eta_R\le 1$ and $|\nabla\eta_R|\le\frac{C}{R}$.\\
\begin{Lemma} Let $v_\alpha$ be  Palais-Smale sequence for $J_\alpha$ at level  $\beta$
that weakly, but not strongly, converges to $0$ in $H_1^2(M)$. Then,
there exists a sequence of positive reals $R_\alpha\to 0 $ such
that, up to a subsequence, $ \hat{\eta}_\alpha\hat{v}_\alpha$ with
\begin{equation*}
\hat{v}_\alpha(x)=R_\alpha^{\frac{n-2}{2}}v_\alpha(
\exp_{p}(R_\alpha x)),
\end{equation*}
and $\hat{\eta}_\alpha(x)=\eta_{\delta}(R_\alpha x))$ ($\delta$ is
some positive constant),  converges weakly in $D_1^2(\mathbb{R}^n)$
to a function $v\in D_1^2(\mathbb{R}^n)$ such that, if $v\neq0, v$
is weak solution of the Euclidean equation
\begin{equation}\label{4.14}
\Delta v-\frac{h_\infty(p)}{|x|^2}v=f(p)|v|^{2^*-2}v.
\end{equation}
\end{Lemma}
\begin{proof} Since the Palais-Smale sequence $v_\alpha$ of $J_\alpha$ at
level $\beta$ converges  weakly and not strongly  in $H_1^2(M)$ to
$0$, we get by Lemma 4.3  that $\beta\ge\beta^*$.\\
Write
\begin{equation*}
\int_M(|\nabla v_\alpha|^2-\frac{h_\alpha}{\rho^2_p}v_\alpha^2)dv_g=
\int_Mf|v_\alpha|^{2^*}dv_g+o(1)=n\beta+o(1),
\end{equation*}
since, up to a subsequence, $v_\alpha$ converges strongly to $0$ in
$L_2(M)$, we get by Hardy inequality \eqref{2.1} that for all
$\varepsilon>0$ small
$$n\beta^*+o(1)\le\int_M|\nabla v_\alpha|^2dv_g\le\frac{n\beta}{1-(h_\alpha(p)+\varepsilon)(K^2(n,2,-2)+\varepsilon)}+o(1).$$
In other words,
\begin{equation}\label{eq3.6}
c_1\le\int_M|\nabla v_\alpha|^2dv_g\le c_2,
\end{equation}
for some positive constants $c_1$ and $c_2$.\\
Let $\hat\delta$ a small positive constant such that
\begin{equation}\label{eq3.7}
    \underset{\alpha\to\infty}{\lim\sup}\int_M|\nabla v_\alpha|^2>\gamma.
\end{equation}
Up to a subsequence, for each $\alpha>0,$ we can find the smallest
constant $r_\alpha>0$  such that
$$\int_{B(p,r_\alpha)}|\nabla v_\alpha|^2dv_g=\gamma.$$
For a sequence of positive constants $R_\alpha$ and $x\in
B(R_\alpha^{-1}\delta_g)\subset\mathbb{R}^n$, define
\begin{eqnarray*}
  \hat{v}_\alpha(x) &=& R_\alpha^{\frac{n-2}{2}}v_\alpha(\exp_{p}(R_\alpha
  x)), \emph{and }
   \\
 \hat{g}_\alpha(x) &=& (\exp_{p}^*g)(R_\alpha x)).
\end{eqnarray*}
We follow the same arguments as in \cite{Druet-hebbey-robert}. Let
$z\in\mathbb{R}^n$ be such that $|z|+r<\delta_gR_\alpha^{-1}$, then
we have
\begin{equation*}
\int_{B(z,r)}|\nabla\hat{v}_\alpha|^2 dv_{\hat{g}}=\int_{\exp_{p}(
R_\alpha B(z,r))}|\nabla v_\alpha|^2 dv_g.
\end{equation*}
Let $0<r_o<\frac{\delta_g}{2}$ be such that for any $x,y\in
B(r_o)\subset\mathbb{R}^n$, the following inequality holds
\begin{equation}\label{4.15}
dist_g(\exp_p(x),\exp_p(y))\le C_o|x-y|,
\end{equation}
 for some positive
constant $C_o$. Also, for $r\in(0,r_o)$, take $R_\alpha$ be such
that $c_orR_\alpha=r_\alpha$, then we get
\begin{equation*}
\exp_p(R_\alpha B(C_or)))=B(p,C_orR_\alpha)
\end{equation*}
 and then
\begin{equation}\label{4.16}
\int_{B(C_or)}|\nabla \hat{v}_\alpha|^2 dv_{\hat{g}}=\int_{
B(p,r_\alpha)}|\nabla v_\alpha|^2 dv_g=\gamma.
\end{equation}
Take $\delta$ such that $0<\delta \le\min(C_or,\frac{\delta_g}{2})$,
there exists a positive constant such that, for all $u\in
D^{1,2}(\mathbb{R}^n)$ with $Supp( u)\in B(\delta R_\alpha^{-1})$,
the following inequalities hold
\begin{eqnarray}
\label{4.17}\frac{1}{C_1}\int_{\mathbb{R}^n}|\nabla u|^2dx \le
\int_{\mathbb{R}^n}|\nabla u|^2 dv_{\hat{g}}&\le&C_1\int_{\mathbb{R}^n}|\nabla u|^2dx, \emph{and}\\
 \label{4.18} \frac{1}{C_1}\int_{\mathbb{R}^n}|u|dx  \le\int_{\mathbb{R}^n}|u|
  dv_{\hat{g}}&\le&{C_1}\int_{\mathbb{R}^n}|u|dx
\end{eqnarray}
Define a sequence of cut-off functions $\hat{\eta}_\alpha$ by
$\hat{\eta}_\alpha(x)= \eta_{\delta}(R_\alpha x)$. Then, it follows
from \eqref{4.16}, \eqref{4.17} and \eqref{4.18} that the sequence
$\tilde{v}_\alpha=\hat{\eta}_\alpha \hat{v}_\alpha$ is bounded in
$D^{1,2}(\R^n)$. Consequently, up to a subsequence,
$\tilde{v}_\alpha$ converges weakly to some
function $v\in D^{1,2}(\R^n)$.\\
Suppose that $v\neq0$, since $v_\alpha$ converges weakly
to $0$, it follows that $R_\alpha\to 0$.\\
Let us first prove that  $v$ is a weak solution on $D^{1,2}(\R^n )$
to \eqref{4.14}. For this task, we let
 $\varphi\in\C^\infty_o(\mathbb{R}^n)$ be a function with compact
support included in the ball $B(\delta)$. For $\alpha$ large, define
on $M$ the sequence $\varphi_\alpha$ as
$$\varphi_\alpha(x)=R_\alpha^{\frac{2-n}{2}}\varphi(R_\alpha^{-1}(\exp^{-1}_{p}( x))).$$
Then, we have
\begin{eqnarray*}
  \int_M\nabla v_\alpha\nabla \varphi_\alpha dv_g &=&
  \int_{\mathbb{R}^n}\nabla\tilde{v}_\alpha\nabla\varphi dv_{\hat{g}_\alpha}, \\
  \int_M\frac{h_\alpha}{\rho_p^2}v_\alpha\varphi_\alpha dv_g &=&
  R_\alpha^{2}\int_{\mathbb{R}^n}\frac{h_\alpha(\exp_{p}(R_\alpha x) )}
  {dist_{\hat{g}_\alpha}^2(0, R_\alpha x)}\tilde{v}_\alpha\varphi dv_ {\hat{g}_\alpha}, \emph{ and }\\
   \int_M f|v_\alpha|^{2^*-2}v_\alpha\varphi_\alpha
   dv_g&=&\int_{\mathbb{R}^n}f(\exp_{p}
   ( R_\alpha x))|\tilde{v}_\alpha|^{2^*-2}\tilde{v}_\alpha\varphi
   dv_ {\hat{g}_\alpha}.
\end{eqnarray*}
When tending $\alpha$ to $\infty$,  $\hat{g}_\alpha$ tends smoothly
to the Euclidean metric on $\R^n$, then by passing to the limit when
$\alpha\to\infty$ and since $v_\alpha$ is a Palais-Smale sequence of
$J_\alpha$, we get that $v$ is weak solution  of \eqref{4.14}.
\end{proof}
\begin{Lemma} Let $v$ be the solution  of \eqref{4.14} given by Lemma 3.4, then
up to a subsequence,
\begin{equation*}
w_\alpha=v_\alpha-R_\alpha^{\frac{2-n}{2}}\eta_\delta(\exp^{-1}(x))
v(R_\alpha^{-1}\exp_{p}^{-1}(x)),
\end{equation*}
 where
$0<\delta<\frac{\delta_g}{2}$,  is a Palais-Sequence for $J_\alpha$
at level $\beta-G_\infty(v)$ that weakly converges to $0$ in
$H_1^2(M).$
\end{Lemma}
\begin{proof}
 For $0<\delta<\frac{\delta_g}{2}$, define
\begin{equation*}
\B_\alpha(x)=R_\alpha^{\frac{2-n}{2}}\eta_\delta(\exp_p^{-1}(x))
v(R_\alpha^{-1}\exp_p^{-1}( x)),x\in M
\end{equation*}
and put
\begin{equation*}
w_\alpha=v_\alpha-\B_\alpha.
\end{equation*}
We begin proving that  $w_\alpha$ converges weakly to $0$ in
$H^2_1(M)$, it suffices to prove that $\B_\alpha$ does. Take a
function $\varphi\in\C^{\infty}(M)$, then we have
\begin{eqnarray*}
&\int_{B(p,2\delta)}\left(\nabla \B_\alpha\nabla\varphi+
\B_\alpha\varphi\right)dv_g&\\&=
R_\alpha^{\frac{n}{2}}\int_{B(2\delta R_\alpha^{-1})}[R_\alpha
v(x)(\nabla\eta_\delta)(R_\alpha x)+\eta_\delta(R_\alpha x)\nabla
v]\nabla \varphi(\exp_p(R_\alpha x))
dv_{\hat{g}_\alpha}&\\&+R_\alpha^{\frac{n+2}{2}}\int_{B(2\delta
R_\alpha^{-1})} v\eta_\delta(R_\alpha x)\varphi(\exp_p(R_\alpha
x))dv_{\hat{g}_\alpha},&
\end{eqnarray*}
then, for a positive constant $C'$ such that $dv_{\hat{g}_\alpha}\le
C'dx$, it follows that
\begin{eqnarray*}
&\int_{B(p,2\delta)}\left(\nabla \B_\alpha\nabla\varphi+
\B_\alpha\varphi\right)dv_g\\&\le
C'R_\alpha^{\frac{n}{2}}[\sup_M|\nabla\varphi|\int_{\R^n}(|\nabla
v|+|v|C\delta^{-1})dx+ R_\alpha\sup_M|\varphi|\int_{\R^n}|v|)dx].&
\end{eqnarray*}
 Thus, when tending $\alpha\to \infty$, we ge that $\B_\alpha\to0 $ weakly in $H^2_1(M)$.\\
Now, let us evaluate $J_\alpha(w_\alpha)$. First, we have
\begin{eqnarray*}
\int_{M}|\nabla w_\alpha|^2dv_g&=&\int_{M\setminus
B(p,2\delta)}|\nabla v_\alpha|^2dv_g+\int_{B(p,2\delta)}|\nabla
(v_\alpha-\B_\alpha)|^2dv_g,
\end{eqnarray*}
and of course
\begin{eqnarray*}
  &&\int_{B(p,2\delta)}|\nabla
(v_\alpha-\B_\alpha)|^2dv_g\\ &=& \int_{B(p,2\delta)}|\nabla
v_\alpha|^2dv_g-2\int_{B(p,2\delta)}\nabla v_\alpha\nabla\B_\alpha
dv_g+\int_{B(p,2\delta)}|\nabla \B_\alpha|^2dv_g.
\end{eqnarray*}
Direct calculation gives
\begin{eqnarray*}
&\\&\int_{B(p,2\delta)}|\nabla \B_\alpha|^2dv_g=\int_{B( 2\delta
R_\alpha^{-1})}\eta^2_\delta(R_\alpha x)|\nabla
v|^2dv_{\hat{g}_\alpha}+&\\&R_\alpha^2\int_{B( 2\delta
R_\alpha^{-1})}v^2|\nabla\eta_{\delta }|^2(R_\alpha
x)dv_{\hat{g}_\alpha}+2R_\alpha\nabla\eta_{\delta }(R_\alpha
x)\nabla vdv_{\hat{g}_\alpha}.&
\end{eqnarray*}
It can be easily seen that the second term of right-hand side member
of the above equality tends to $0$ as $\alpha\to \infty $.
Furthermore,  for $R>0$, a positive constant, we write
\begin{equation*}
\int_{B( 2\delta R_\alpha^{-1})}\eta^2_\delta(R_\alpha x)|\nabla
v|^2dv_{\hat{g}_\alpha}=\int_{B( R)}\eta^2_\delta(R_\alpha x)|\nabla
v|^2dv_{\hat{g}_\alpha}+\int_{\mathbb{R}^n\setminus
B(R)}\eta^2_\delta(R_\alpha x)|\nabla v|^2dv_{\hat{g}_\alpha}.
\end{equation*}
 with
\begin{equation*}
\int_{\mathbb{R}^n\setminus B(R)}\eta^2_\delta(R_\alpha x)|\nabla
v|^2dv_{\hat{g}_\alpha}\le C \int_{\mathbb{R}^n\setminus
B(R)}|\nabla v|^2dx=\varepsilon_R,
\end{equation*}
where $\varepsilon_R$ is a function in $R$ such that $ \varepsilon_R\to 0 $ as $R\to\infty$.\\
Noting that, that ${\hat{g}_\alpha}$ goes locally in $C^1$ to the
Euclidean metric $\xi$, we get then
\begin{equation}
\int_{B(p,2\delta)}|\nabla
\B_\alpha|^2dv_g=\int_{\mathbb{R}^n}|\nabla
v|^2dx+o(1)+\varepsilon_R.
\end{equation}\label{3.15}
Moreover, we have
\begin{eqnarray}\label{3.16}
&\int_{B(p,2\delta)}\nabla v_\alpha\nabla\B_\alpha
dv_g=\int_{B(2\delta R_\alpha^{-1})}\nabla(\eta_\delta(R_\alpha x)
\hat{v}_\alpha)\nabla vdv_{\hat{g}_\alpha}
&\\&\nonumber+R_\alpha\int_{B(2\delta R_\alpha^{-1})}(v\nabla
\hat{v}_\alpha -\hat{v}_\alpha\nabla v)\nabla\eta_{\delta}(R_\alpha
x) dv_{\hat{g}_\alpha}&
\end{eqnarray}
with
\begin{eqnarray*}
 && |\int_{B(2\delta
R_\alpha^{-1})}\nabla\eta_\delta(R_\alpha x)(v\nabla \hat{v}_\alpha
-\hat{v}_\alpha\nabla v)dv_{\hat{g}_\alpha}|\\
 &\le&
c\delta^{-1}\left[\right.\int_{B(2\delta R_\alpha^{-1})} |\nabla
\hat{v}_\alpha|^2 dv_
{\hat{g}_\alpha})^{\frac{1}{2}}(\int_{B(2\delta R_\alpha^{-1})} v^2
dx)^{\frac{1}{2}}\\  &+&(\int_{B(2\delta
R_\alpha^{-1})}\hat{v}_\alpha^2
dv_{\hat{g}_\alpha})^{\frac{1}{2}}(\int_{B(2\delta R_\alpha^{-1}))}
|\nabla v|^2 dx)^{\frac{1}{2}}\left.\right].
\end{eqnarray*}
Since $v_\alpha$ is bounded in $H^2_1(M)$, the quantities
$\int_{B(2\delta R_\alpha^{-1})} |\nabla \hat{v}_\alpha|^2 dv_
{\hat{g}_\alpha}$ and$\int_{B(2\delta R_\alpha^{-1})} |
\hat{v}_\alpha|^2 dv_ {\hat{g}_\alpha}$ are bounded and hence the
second term of the right-hand side member of \eqref{3.16} is $o(1)$.
Thus, by using the weak convergence of
$\hat{\eta}_\alpha\hat{v}_\alpha$ to $v$ in $D^{1,2}(\R^n)$ that
\begin{equation*}
\int_{B(p,\delta)}\nabla v_\alpha\nabla\B_\alpha
dv_g=\int_{\mathbb{R}^n}|\nabla v|^2dx+o(1).
\end{equation*}
so that
\begin{equation*}
\int_{M}|\nabla w_\alpha|^2dv_g=\int_{M}|\nabla
v_\alpha|^2dv_g-\int_{\mathbb{R}^n}|\nabla
v|^2dx+o(1)+\varepsilon_R.
\end{equation*}
In the same vain, for $R$ a positive constant and $\alpha$ large, we
write
\begin{equation*}
\int_{B(p,2\delta)}\frac{h_\alpha}{\rho_p^2}\B_\alpha^2dv_g=
\int_{B(p,RR_\alpha)}\frac{h_\alpha}{\rho_p^2}\B_\alpha^2dv_g+
\int_{B(p,2\delta)\setminus
B(p,RR_\alpha)}\frac{h_\alpha}{\rho_p^2}\B_\alpha^2dv_g
\end{equation*}
with
\begin{equation*}
\int_{B(p,2\delta)\setminus
B(p,RR_\alpha)}\frac{h_\alpha}{\rho_p^2}\B_\alpha^2dv_g\le C
(RR_\alpha)^{-2}\int_{B(p,2\delta)\setminus
B(p,RR_\alpha)}\B_\alpha^2dv_g
\end{equation*}
then, by a direct calculations, we get
\begin{equation*}
\int_{B(p,2\delta)\setminus
B(p,RR_\alpha)}\frac{h_\alpha}{\rho_p^2}\B_\alpha^2dv_g\le C
R^{-2}\int_{\R^n\setminus B(R)}v^2dx=\varepsilon_R.
\end{equation*}
Hence,
\begin{eqnarray*}
\int_{B(p,2\delta)}\frac{h_\alpha}{\rho_p^2}\B_\alpha^2&=&
R_\alpha^2\int_{B(R)} \frac{h_\alpha(\exp_p(R_\alpha
x))}{(dist_{\hat{g}_\alpha}(0,R_\alpha x)^2}
\eta^2_\alpha(R_\alpha x) v^2dv_{\hat{g}_\alpha}+\varepsilon_R\\
   &=&h_\infty(p)\int_{\mathbb{R}^n}\frac{v^2}{|x|^2}dx+o(1)+\varepsilon_R.
\end{eqnarray*}
Also, in similar way, since $v_\alpha$ is bounded in $H^2_1(M)$,
after using Hölder and Hardy inequalities, we can easily have
\begin{equation*}
\int_{B(p,2\delta)\setminus
B(p,RR_\alpha)}\frac{h_\alpha}{\rho_p^2}v_\alpha\B_\alpha dv_g\le C
R^{-2}\int_{\R^n\setminus B(R)}v^2dv_g=\varepsilon_R,
\end{equation*}
which yields
\begin{eqnarray*}
\int_{B(p,\delta)}\frac{h_\alpha}{\rho_p^2}v_\alpha\B_\alpha dv_g&=&
R_\alpha^2\int_{B(R)} \frac{h_\alpha(\exp_p(R_\alpha
x))}{(dist_{\hat{g}_\alpha}(0,R_\alpha x))^2}
(\eta(R_\alpha x)\hat{v}_\alpha) vdv_{\hat{g}_\alpha}+\varepsilon_R\\
   &=&h_\infty(p)\int_{\mathbb{R}^n}\frac{v^2}{|x|^2}dx+o(1)+\varepsilon_R.
\end{eqnarray*}
so that in the end we obtain
\begin{equation*}
\int_M\frac{h_\alpha}{\rho_p^2}w_\alpha^2dv_g=
\int_M\frac{h_\alpha}{\rho_p^2}v_\alpha^2dv_g-h_\infty(p)\int_{\mathbb{R}^n}\frac{v^2}{|x|^2}dx+o(1)+\varepsilon_R.
\end{equation*}
In similar way, we can prove that
\begin{equation*}
\int_M|w_\alpha|^{2^*}dv_g=\int_M|v_\alpha|^{2^*}dv_g-f(p)\int_M|v|^{2^*}dv_g+o(1)+\varepsilon_R,
\end{equation*}
Finally, since $R$ is arbitrary,  when summing up we obtain
\begin{equation*}
J_\alpha(w_\alpha)=J_\alpha(u_\alpha)-G_\infty(v)+o(1)=\beta-G_\infty(v)+o(1).
\end{equation*}
It remains to prove that  $DJ_\alpha(\B_\alpha)\to 0$ in
$H_1^2(M)'$. Let $\varphi\in
 H_1^2(M)$, for $x\in B(\delta R_\alpha^{-1})$ put $\varphi_\alpha(x)=R_\alpha^{\frac{n-2}{2}}\varphi(\exp_p(R_\alpha
x))$ and $\overline{\varphi}_\alpha(x)=\eta_\delta(R_\alpha
x))\varphi_\alpha (x)$, then we have
\begin{eqnarray*}
\int_{B(p,2\delta)} \nabla \B_\alpha\nabla\varphi
dv_g=\int_{B(2\delta R_\alpha^{-1})} \nabla
v\nabla\overline{\varphi}_\alpha
dv_{\hat{g}_\alpha}&&\\+R_\alpha\int_{B(2\delta R_\alpha^{-1})}
\nabla\eta_\delta(R_\alpha
x)(v\nabla\varphi_\alpha-\varphi_\alpha\nabla v)dv_{\hat{g}_\alpha}.
  \end{eqnarray*}
Knowing that $\int_{B(p,2\delta)} |\nabla\varphi|^2
dv_g=\int_{B(2\delta R_\alpha^{-1})}  |\nabla\varphi_\alpha|^2
dv_{\hat{g}_\alpha}$, we get that
\begin{equation*}
\int_{B(2\delta R_\alpha^{-1})}| \nabla\eta_\delta(R_\alpha
x)(v\nabla\varphi_\alpha-\varphi_\alpha\nabla
v)|dv_{\hat{g}_\alpha}\le C||\varphi||_{H^2_1(M)},
\end{equation*}
which gives that
\begin{equation*}
\int_{B(p,2\delta)} \nabla \B_\alpha\nabla\varphi
dv_g=\int_{B(2\delta R_\alpha^{-1})} \nabla
v\nabla\overline{\varphi}_\alpha dv_{\hat{g}_\alpha}+o(||\varphi
||_{H_1^2(M)}).
\end{equation*}
Next, for $R>0$ write
\begin{equation*}
\int_{B(2\delta R_\alpha^{-1})} \nabla
v\nabla\overline{\varphi}_\alpha dv_{\hat{g}_\alpha}=\int_{B(R)}
\nabla v\nabla\overline{\varphi}_\alpha
dv_{\hat{g}_\alpha}+\int_{B(2\delta R_\alpha^{-1})\setminus B(R)}
\nabla v\nabla\overline{\varphi}_\alpha dv_{\hat{g}_\alpha},
\end{equation*}
note that
\begin{eqnarray*}
  \int_{B(2\delta R_\alpha^{-1})\setminus B(R)}
\nabla v\nabla\overline{\varphi}_\alpha dv_{\hat{g}_\alpha}&\le&
C||\varphi||_{H^2_1(M)}(\int_{B(2\delta R_\alpha^{-1})\setminus
B(R)} |\nabla v|^2 dx)^\frac{1}{2}\\&=& O(||\varphi||_{H^2_1(M)}
)\varepsilon(R),
\end{eqnarray*}
where  $\varepsilon_R\to 0$ as $R\to \infty$. Since the sequence of
metrics ${\hat{g}_\alpha} $ tends locally in $C^1$ when
$\alpha\to\infty$ to the Euclidean metric, we obtain
\begin{equation*}
\int_{B(p,2\delta)} \nabla \B_\alpha\nabla\varphi dv_g=\int_{\R^n}
\nabla v\nabla\overline{\varphi}_\alpha dx+o(||\varphi
||_{H_1^2(M)})+O(||\varphi||_{H^2_1(M)} )\varepsilon(R).
\end{equation*}
Moreover, for a given  $R>0$, we have for $\alpha$ large,
\begin{equation*}
\int_{B(p,2\delta)}\frac{h_\alpha}{\rho_p^2}\B_\alpha\varphi dv_g=
\int_{B(p,RR_\alpha)}\frac{h_\alpha}{\rho_p^2}\B_\alpha\varphi
dv_g+\int_{B(p,2\delta)\setminus
B(p,RR_\alpha)}\frac{h_\alpha}{\rho_p^2}\B_\alpha\varphi dv_g.
\end{equation*}
 On the one hand, we have
\begin{equation*}
\int_{B(p,2\delta)\setminus
B_p(RR_\alpha)}\frac{h_\alpha}{\rho_p^2}\B_\alpha\varphi dv_g\le
\frac{C}{(RR_\alpha)^2}||\varphi||_{H_1^2(M)}\int_{B(p,2\delta)\setminus
B(p,RR_\alpha)}\B_\alpha^2dv_g,
\end{equation*}
and a straightforward computation  shows that
\begin{equation*}
\int_{B(p,2\delta)\setminus B(p,RR_\alpha)}|\B_\alpha|^2dv_g\le C
R_\alpha^2\int_{B(2\delta R_\alpha^{-1})\setminus B(R)}v^2dx,
\end{equation*}
  which implies that
\begin{equation*}
   \int_{B(p,2\delta)\setminus
B(p,RR_\alpha)}\frac{h_\alpha}{\rho_p^2}\B_\alpha\varphi dv_g=O(
||\varphi||_{H_1^2(M)})\varepsilon_R
\end{equation*}
with
$\varepsilon_R \to 0$ as $R\to \infty$. \\
On the other hand,  we have
\begin{equation*}
\int_{B(p,RR_\alpha)}\frac{h_\alpha}{\rho_p^2}\B_\alpha\varphi dv_g=
R_\alpha^{2}\int_{B(R)}\frac{h_\alpha(\exp_pR_\alpha
x)}{(dist_{\hat{g}_\alpha}(0,R_\alpha
x))^2}v\overline{\varphi}dv_{\hat{g}}.
\end{equation*}
which leads to
\begin{eqnarray*}
  \int_{B_p(RR_\alpha)}\frac{h_\alpha}{\rho_p^2}\B_\alpha\varphi
dv_g &=&\int_{B(R)}\frac{h_\infty(p)
}{|x|^2}v\overline{\varphi}dx+o(||\varphi ||_{H_1^2(M)}) \\
   &=&\int_{\mathbb{R}^n}\frac{h_\infty(p)
}{|x|^2}v\overline{\varphi}dx-\int_{\mathbb{R}^n\setminus
B(R)}\frac{h_\infty(p) }{|x|^2}v\overline{\varphi}dx +o(||\varphi
||_{H_1^2(M)}),
\end{eqnarray*}
with
\begin{eqnarray*}
\int_{\mathbb{R}^n\setminus B(R)}\frac{h_\infty(p)
}{|x|^2}v\overline{\varphi}dx&\le&
\frac{C}{R^2}||\varphi||_{H_1^2(M)}\\ &=&O(
||\varphi||_{H_1^2(M)})\varepsilon_R.
\end{eqnarray*}
so that \begin{equation*}
   \int_{B(p,2\delta)}\frac{h_\alpha}{\rho_p^2}\B_\alpha\varphi dv_g= \int_{\mathbb{R}^n}\frac{h_\infty(p)
}{|x|^2}v\overline{\varphi}dx+o(||\varphi ||_{H_1^2(M)})+O(
||\varphi||_{H_1^2(M)})\varepsilon_R.
\end{equation*}
In the same way, we can also have
\begin{equation*}
    \int_{B(p,2\delta)}f|\B_\alpha|^{\frac{4}{n-2}}\B_\alpha\varphi dv_g =
f(p)\int_{\mathbb{R}^n}|v|^{\frac{4}{n-2}}v\overline{\varphi}_\alpha
dx+o(||\varphi ||_{H_1^2(M)}) +O(||\varphi
||_{H_1^2(M)})\varepsilon_R.
\end{equation*}
Summing up, we obtain
\begin{eqnarray*}
&\int_{B(p,2\delta)}(\nabla\B_\alpha\nabla \varphi
dv_g+\frac{h_\alpha}{\rho_p^2}\B_\alpha\varphi)dv_g-
\int_{B(p,2\delta)}f|\B_\alpha|^{\frac{4}{n-2}}\B_\alpha\varphi
dv_g& \\&= \int_{\R^n}( \nabla v\nabla\overline{\varphi}_\alpha
dx+\frac{h_\infty(p) }{|x|^2}v\overline{\varphi}_\alpha)dx-
  f(p)\int_{\mathbb{R}^n}|v|^{\frac{4}{n-2}}v\overline{\varphi}_\alpha dx&\\&+o(||\varphi
||_{H_1^2(M)}) +O(||\varphi ||_{H_1^2(M)})\varepsilon_R,&
\end{eqnarray*}
and since $v$ is  weak solution of \eqref{0.1}, we get the desired
result.
\end{proof}
Keeping the notations adapted above, we prove the following lemma
\begin{Lemma} Let $v_\alpha$ a Palais-Smale sequence for $J_\alpha$ at level $\beta$.
Suppose that the sequence
$\tilde{v}=\hat{\eta}_\alpha\hat{v}_\alpha$ of the above lemma
converges weakly to $0$ in $D^{1,2}(\R^n)$. Then, there exist a
sequence of positive numbers $\{\t_\alpha\},\t_\alpha\to0$ and a
sequence of points $x_i\in M, x_i\to x_o\in M\setminus\{p\}$ such
that up to a subsequence, the sequence $\eta_\delta(\t_\alpha
x)\nu_\alpha$, with $\delta$ is some constant and
\begin{equation*}
 \nu_\alpha=\t_\alpha^{\frac{n-2}{2}} v_\alpha(\exp_{x_i}(\t_\alpha x ))
\end{equation*}
 converges weakly to a nontrivial weak solution $\nu$ of the Euclidean
equation
\begin{equation}\label{eq3.8}
\Delta\nu=f(x_o)|\nu|^{\frac{4}{n-2}}\nu
\end{equation}
 and the sequence
\begin{equation*}
    \W_\alpha=v_\alpha-\t_\alpha^{\frac{2-n}{2}}\eta_\delta(\exp_{x_i}^{-1}(x))\nu(\t_\alpha^{-1}\exp_{x_i}^{-1}(x))
\end{equation*}
is a Palais-Smale sequence for $J_\alpha $ at level
$\beta-f(x_o)^{\frac{4}{n-2}}G(\nu)$ that converges weakly to $0$ in
$H^1_2(M)$.
\end{Lemma}
\begin{proof} Suppose that the sequence
$\tilde{v}_\alpha=\hat{\eta}_\alpha \hat{v}_\alpha$ converges weakly
to $0$ in $\dob(\R^n) $ . Take a function $\varphi\in
\C^{\infty}_o(B(C_or))$ and put
$\varphi_\a(x)=\varphi(R_\alpha^{-1}\exp_p^{-1}(x)) $. As in
\cite{D. Smet} and \cite{D.Cao}, by the strong convergence of
$\tilde{v}_\alpha$ to $0$ in $L_{loc}^2(\R^n)$, we have for $\alpha$
large
\begin{eqnarray}\label{4.19}
 &&\nonumber\int_{\R^n}|\nabla( \tilde{v}_\alpha\varphi)|^2dv_{\hat{g}_\alpha}=\int_{\R^n}\nabla
\tilde{v}_\alpha\nabla(\tilde{v}_\alpha\varphi^2)dv_{\hat{g}_\alpha}+o(1)
 \\ &=&\int_{M}\nabla v_\alpha\nabla
(v_\alpha\varphi_\alpha^2)dv_{g}+o(1)\nonumber\\ &=&\|\mathcal{D}
J_\alpha\| \|{v_\alpha\varphi_\alpha^2}\|+
  \int_{M}\frac{h_\alpha}{\rho_p^2}(v_\alpha\varphi_\a)^2dv_g+
  \int_{M}f|v_\alpha|^{\frac{4}{n-2}}(v_\alpha\varphi_\alpha)^2dv_g+o(1)\\
   &\le&(h_\alpha(p)+\varepsilon)(K^2(n,2,-2)+\varepsilon)\int_{\R^n}|\nabla
  (\tilde{v}_\alpha\varphi)|^2dv_{\hat{g}_\alpha}+\nonumber\\ &&\sup_M
   f K^{2^*}(n,2)(\int_{B(C_or)}|\nabla \tilde{v}_\alpha|^2dv_{\hat{g}_\alpha})^{\frac{2}{n-2}}\int_{\R^n}|\nabla
   (\tilde{v}_\alpha\varphi)|^2dv_{\hat{g}_\alpha}+o(1)\nonumber.
\end{eqnarray}
Thus, for $\gamma$ chosen small enough,  we get that for each
$t,0<t<C_or,$
\begin{equation}\label{4.20}
\int_{B(p,tR_\alpha)}|\nabla v_\alpha|^2dv_g=\int_{B(t)}|\nabla
\tilde{v}_\alpha|^2dv_{\hat{g}}\to0, \text{ as }\alpha\to \infty.
\end{equation}
Now, for $t>0$ consider the function
\begin{equation*}
t\longrightarrow\F(t)=\max_{x\in M}\int_{B(x,t)}|\nabla
v_\alpha|^2dv_g.
\end{equation*}
Since $\F $ is continuous, under \eqref{eq3.6} and \eqref{eq3.7}, it
follows that for any $\lambda\in(0,\gamma)$, there exist
$t_\alpha>0$ small and $x_\alpha\in M$ such that
\begin{equation*}
\int_{B(x_\alpha,t_\alpha)}|\nabla v_\alpha|^2dv_g=\lambda.
\end{equation*}
Since $M$ is compact, up to a subsequence, we may assume that
$x_\alpha$ converges to some point $x_o\in M$.\\
Note first that for all  $\alpha\ge0$,
$t_{\alpha}<r_{\alpha}=C_orR_{\alpha}$, otherwise if there exists
$\alpha_o\ge 0$ such that $t_{\alpha_o}<r_{\alpha_o}$, we get a
contradiction due to the fact that
\begin{equation*}
\lambda= \int_{B(x_{\alpha_o},t_{\alpha_o})}|\nabla
v_{\alpha_o}|^2dv_g\ge\int_{B(p,t_{\alpha_o})}|\nabla
v_{\alpha_o}|^2dv_g\ge\int_{B(p,r_{\alpha_o})}|\nabla
v_{\alpha_o}|^2dv_g =\gamma.
\end{equation*}
\\ Now, suppose that for all $\varepsilon>0$, there exists
$\alpha_\varepsilon>0$ such that $dist_g(x_\alpha,p)\le\varepsilon$
 for all $\alpha\ge\alpha_\varepsilon$. Choose $r'_\alpha$ such that,  $t_{\alpha}<r'_\alpha<r_{\alpha} $ and take $\varepsilon'
=r'_\alpha-t_\alpha$, we get that for some $\alpha_{\varepsilon'}>0$
and $\alpha\ge\alpha_{\varepsilon'}$
\begin{equation*}
 B(x_\alpha, t_\alpha)\subset B(p,r'_\alpha),
\end{equation*}
which, by virtue of \eqref{4.20}, is impossible. We deduce then that
$x_o\neq p$. \\
Now, let $0<\t_\alpha<1$, for $x \in
B(\t_\alpha^{-1}\delta_g)\subset\mathbb{R}^n $ consider the
sequences
\begin{eqnarray*}
  \nu_\alpha(x) &=&\t_\alpha^{\frac{n-2}{2}}v_\alpha(\exp_{x_\alpha}(\t_\alpha
  x)), \text{ and }
   \\
 \tilde{g}_\alpha(x) &=& \exp_{x_\alpha}^*g(\t_\alpha x)).
\end{eqnarray*}
Take $\t_\alpha$ such that $C_or\t_\a=t_\alpha$. As in the above
lemma, we can easily check that there is a subsequence of
$\hat{\nu}_\alpha=\eta_{\delta}(\t_\alpha x) \nu_\alpha$ where
$\delta$ is as in the above lemma,  that weakly converges  in
$\dob(\R^n)$ to some function $\nu$, a weak solution on $\dob(\R^n)$
to \eqref{eq3.8}. Note that this time the
singular term disappears because $x_o\neq p$  and because of course $t_\alpha\to0$.\\
It remains to show that $\nu\neq0$. For this purpose, take a point
$a\in\R^n$ and a constant $r>0$ such that $|a|+r<r_o\t_\alpha^{-1}$,
where $r_o\in (0,\frac{\delta_g}{2})$ is a constant such that
inequality \eqref{4.15} is satisfied. Then, we have
\begin{equation*}
 \exp_{x_\alpha}(\t_\alpha B(a,r))\subset B(\exp_{x\alpha}(\t_\alpha
a),C_or\t_\alpha),
\end{equation*}
and
\begin{equation*}
   \exp_{x_\alpha}(\t_\alpha B(C_or))= B(x_\alpha,C_or\t_\alpha)
\end{equation*}
$C_o$, here, is the constant appearing in inequality \eqref{4.15}.
Since we have
\begin{equation*}
\int_{B(a,r)}|\nabla
\nu_\alpha|^2dv_{\tilde{g}_\alpha}=\int_{\exp_{x_\alpha}(\t_\alpha
B(a,r))}|\nabla v_\alpha|^2dv_{g},
\end{equation*}
we get by construction of $x_\alpha$ that for such $a$ and $r$ ,
\begin{equation*}
\int_{B(a,r)}|\nabla \nu_\alpha|^2dv_{\tilde{g}}\le\lambda
\end{equation*}
Suppose now that $\nu\equiv0$. Take any function $h\in \dob(\R^n)$
with support included in a ball $B(a,r)\subset\R^n$, with $a$ and
$r$ as above. Then, by taking $\lambda$ small enough, we get by the
same calculation done in \eqref{4.19} that $\int_{B(a,r)}\nabla
\hat{\nu}_\alpha dv_{\tilde{g}}$ converges to $0$ for all $a\in\R^n$
and $r>0$ such that  $|a|+r<r_o\t_\alpha^{-1}$. In particular,
\begin{equation*}
\int_{B(x_\alpha,t_\alpha)}|\nabla
v_\alpha|^2dv_{g}=\int_{B(C_or)}|\nabla
\nu_\alpha|^2dv_{\tilde{g}}\to0,
\end{equation*}
which makes a contradiction. Thus $\nu\neq0$.\\
The proof of the remaining statements of the lemma goes in the same
way as in lemma 3.5.
\end{proof}
\begin{proof}[Proof of Theorem 4.1] First, it is worthy  to mention that the value
$G_\infty(v)$ taken on a nontrivial  weak solution $v$ of the
Euclidean equation \eqref{4.15} is greater or equal to the constant
$\beta^*$. In fact, if $v$ is solution of \eqref{4.15},then by Hardy
and Sobolev inequalities we have
\begin{equation}\label{4.27}
\int_{\R^n}(|\nabla
v|^2-h_\infty(p)\frac{v^2}{|x|^2})dx=f(p)\int_{\R^n}|v|^{2^*}dx\le
f(p)K^{2^*}(n,2)(\int_{\R^n}|\nabla v|^2dx)^{\frac{2^*}{2}},
\end{equation}
and
\begin{equation}\label{4.28}
\int_{\R^n}(|\nabla v|^2-h_\infty(p)\frac{v^2}{|x|^2})dx\ge
(1-h_\infty(p)K^2(n,-2,2))\int_{\R^n}|\nabla v|^2dx,
\end{equation}
then by \eqref{4.27} and \eqref{4.28} we get
\begin{eqnarray}\label{4.32}
\nonumber G_\infty(v)&=&\frac{1}{n}\int_{\R^n}(|\nabla
v|^2-h_\infty(p)\frac{v^2}{|x|^2})dx\\
&\ge&\frac{(1-h_\infty(p)K^2(n,-2,2))^{\frac{n}{2}}}{nf(p)^{\frac{n-2}{2}}K^{n}(n,2)}=\beta^*.
\end{eqnarray}
Now, let $u_\alpha$ a sequence of solutions of \eqref{eq1.1} such
that $\int_Mf|u_\alpha|^{2^*}dv_g\le C$,  $u_\alpha$ is then a
bounded Palais-Smale sequence of $J_\alpha$ at some level $\beta$.
Up to a subsequence, we may assume that $u_\alpha$ converges weakly
in $H^2_1(M)$ and almost everywhere in $M$ to a solution $u$ of
\eqref{0.1}. Set $v_\alpha=u_\alpha-u$, then by Lemma 3.2,
$v_\alpha$ is a Palais sequence of $J_\alpha$ at level $\beta_1=
\beta-J_{\infty}(u)+o(1)$. If $v_\alpha\to0$ strongly in $H^2_1(M)$,
then the theorem is proved with $k=l=0$. If $v_\alpha\to0$ only
weakly in $H^2_1(M)$, then we apply Lemmas 3.4, 3.4 and 3.6 to get a
new Palais-Smale sequence $v^1_\alpha$ at level $\beta_2\le
\beta_1-\beta^*+o(1)$. So, either $\beta_2<\beta^*$ and then
$v_\alpha^1$ converges strongly to $0$, or $\beta_2\ge\beta^*$ and
in this case we repeat the procedure for $v_\alpha^1$ to obtain
again a new Palais -Smale sequence at smaller level. By induction,
after a number of iterations, we obtain a Plais-Smale sequence at a
level smaller than $\beta^*$.
\end{proof}
\begin{Corollary}Suppose that the sequence $u_\alpha$ of weak solutions of \eqref{eq1.1} is such that
\begin{equation*}
E(u_\alpha)=\int_Mf|u_\alpha|^{2^*}dv_g\le c\le
\frac{\left(1-h_\infty(p)K^{2}(n,2,-2)\right)^{\frac{n}{2}}}{\underset{M}{(\sup
f)}^{\frac{n-2}{2}}K^{n}(n,2)}.
\end{equation*}
Then, up to a subsequence, $u_\alpha$ converges strongly in
$H^2_1(M)$ to a nontrivial weak solution $u$ of \eqref{0.1}.
\end{Corollary}
\begin{proof} By theorem 4.1, there is a weak solution  $u$ of \eqref{0.1} such that, up to a
subsequence of $u_\alpha$, we have

\begin{eqnarray*}
u_\alpha&=&u+\sum_{i=1}^{k}(R^i_\alpha)^{\frac{2-n}{n}}
\eta_\delta(\exp^{-1}_p(x))v_i((R_\alpha^i)^{-1}\exp^{-1}_p(x))\\&+&\sum_{j=1}^{l}f(x_o^j)^{\frac{2-n}{4}}(r^i_\alpha)^{\frac{2-n}{n}}
\eta_\delta(\exp^{-1}_{x_\alpha^j}(x))\nu_j((r_\alpha^j)^{-1}\exp^{-1}_{x_\alpha^j}(x))+\W_\alpha,\\&&
\text{ with } \W_\alpha\to 0 \text{ in }H^1_2(M),
\end{eqnarray*}
and
\begin{eqnarray*}
c\ge
E(u_\alpha)&=&nJ_\alpha(u_\alpha)\\&=&nJ_{\infty}(u)+n\sum_{i=1}^k
G_\infty(v_i)+n\sum_{j=1}^lf(x_o^j)^{\frac{2-n}{2}} G(\nu_j)+o(1).
\end{eqnarray*}
Suppose that $u\equiv0$, if there exists $i, 1\le i\le  k$ such that
$v_i\neq0$, then by \eqref{4.32} we get
\begin{equation*}
c\ge\frac{\left(1-h_\infty(p)K^{2}(n,2,-2)\right)^{\frac{n}{2}}}{\underset{M}{(\sup
f)}^{\frac{n-2}{2}}K^{n}(n,2)},
\end{equation*}
thus, $v_i\equiv0, \forall i,  1\le i\le  k$, case in which Lemma
4.5 applies, that is, there exists $\nu_j\neq0$ such that
\begin{equation*}
c\ge\frac{f(x_o^j)^{\frac{2-n}{2}}}{K^{n}(n,2)}
 > \frac{\left(1-h_\infty(p)K^{2}(n,2,-2)\right)^{\frac{n}{2}}}{\underset{M}{(\sup
f)}^{\frac{n-2}{2}}K^{n}(n,2)}.
\end{equation*}
Hence, $u\neq0$. Furthermore, $J_\infty(u)>0$ from which we can
conclude that  $k=l=0$. In particular, $u_\alpha$ converges strongly
in $H^2_1(M)$ to $u$.
\end{proof}

\end{document}